\documentclass[a4paper, 12pt, reqno]{amsart}
\usepackage[english]{babel}
\usepackage[utf8]{inputenc}
\usepackage[T1]{fontenc}
\usepackage{amsthm, amsmath, amssymb, amsrefs, enumerate, enumitem, mathtools, lmodern, microtype, listings, tikz-cd}
\usepackage{a4wide}
\usepackage{hyperref}

\theoremstyle{plain}
\newtheorem{thm}{Theorem}[section]
\newtheorem{prop}[thm]{Proposition}
\newtheorem{lemma}[thm]{Lemma}

\theoremstyle{definition}
\newtheorem{defn}[thm]{Definition}
\theoremstyle{remark}
\newtheorem*{rmk}{Remark}
\newtheorem*{rmks}{Remarks}

\binoppenalty=\maxdimen
\relpenalty=\maxdimen

\numberwithin{equation}{section}
\setlist{nosep}
\setlist{noitemsep}

\lstdefinelanguage{Sage}[]{Python}
{morekeywords={False,sage,True},sensitive=true}

\newcommand{\C}{\mathbb{C}}
\renewcommand{\H}{\mathbb{H}}
\newcommand{\Z}{\mathbb{Z}}

\newcommand{\N}{\mathbb{N}}
\newcommand{\R}{\mathbb{R}}
\newcommand{\dd}{\mathrm{d}}
\newcommand{\slz}{{\text {\rm SL}}_2(\mathbb{Z})}
\DeclareMathOperator{\re}{Re}

\newcommand{\vt}[1]{\left\lvert #1 \right\rvert}
\renewcommand{\div}{\mathrm{div}}
\newcommand{\Qc}{\mathcal{Q}}

\title{On a divisor modular form and a theta lift}
\author{Andreas Mono}
\author{Larry Rolen}
\address{Department of Mathematics, 1326 Stevenson Center, Vanderbilt University, Nashville, TN 37240, USA}
\email{andreas.mono@vanderbilt.edu}
\email{larry.rolen@vanderbilt.edu}
\author{Johann Stumpenhusen}
\address{Department Mathematik/Informatik, Abteilung Mathematik, Universit\"at zu K\"oln, Weyertal 86--90, 50931 Cologne, Germany}
\email{jstumpen@math.uni-koeln.de}
\date{\today}

\begin{document}

\begin{abstract}
In $1975$, Zagier introduced the highly influential hyperbolic Poincar{\'e} series $f_{k,D}$. We connect the divisor modular form of $f_{k,D}$ to a new weak Maa{\ss} form $\omega_{k+1,D}$. Furthermore, we show that the generating function of $\omega_{k+1,D}$ has the same modularity properties as Kohnen and Zagier's fruitful theta kernel generating the $f_{k,D}$'s. This yields a new theta lift.
\end{abstract}



\thanks{The first author received support for this research by an AMS--Simons Travel Grant. The second author's work was supported by a grant from the Simons Foundation (853830, LR)}

\maketitle

\section{Introduction and Statement of Results}

Throughout, we let $2 < k \in 2\N$, $D \in \N$ be a discriminant, and
\begin{align*}
\Qc_D \coloneqq \left\{Q(X,Y)= aX^2+bXY+cY^2 \colon a,b,c \in \Z, \ b^2-4ac = D\right\}.
\end{align*}
In 1975, Zagier \cite{zagier75}*{Appendix 2} introduced the function
\begin{align*}
f_{k,D}(z) \coloneqq \sum_{Q \in \Qc_D} \frac{1}{Q(z,1)^k}, \qquad z \in \H \coloneqq \left\{z = x+iy \in \C \colon y > 0\right\},
\end{align*}
which turns out to be a weight $2k$ cusp form for $\slz$ (see Definition \ref{defn:modularforms}). Soon after that, Katok \cite{katok} showed that $f_{k,D}$ can be written as a hyperbolic Poincar{\'e} series by averaging over $z^{-k}$, which is the constant term in the hyperbolic expansion of a modular form. Such expansions go back to Petersson \cite{pet43}, and an excellent survey on such expansions can be found in \cite{IOS}. For this reason, the $f_{k,D}$-functions generate the space of weight $2k$ cusp forms for varying $D$, paralleling the parabolic situation of the usual holomorphic Poincar{\'e} series of exponential type (see \eqref{eq:Pdef}). Bengoechea \cite{beng13} studied the $f_{k,D}$ functions for negative discriminants $D$, which have poles at the corresponding CM points.

Let
\begin{align*}
j(z) \coloneqq q^{-1} + 744 + O(q), \qquad q \coloneqq e^{2\pi i z},
\end{align*}
be Klein's modular invariant for $\slz$ and, for every $n \in \N_0$, let $j_n$ be the unique modular function for $\slz$ having a Fourier expansion of the form $q^{-n} + O(q)$. The first few are explicitly given by
\begin{align*}
j_0(z) = 1, \qquad j_1(z) = j(z) - 744, \qquad j_2(z) = j(z)^2 - 1488j(z) + 159768,
\end{align*}
and the $j_n$ generate the space of weakly holomorphic modular functions for $\slz$ (see Definition \ref{defn:modularforms}). For example, the functions $j_n$ were central to Zagier's \cite{zagier02} work on traces of singular moduli. It is natural to consider the generating function
\begin{align} \label{eq:Hzdef}
H_z(\tau) &\coloneqq \sum_{n \geq 0} j_n(z)e^{2\pi i n \tau}, \qquad \tau = u + iv \in \H.
\end{align}
If $v > y$, Asai, Kaneko, and Ninomiya \cite{askani} established the expansion
\begin{align*}
H_z(\tau) = \frac{\frac{1}{2\pi i} \frac{\partial }{\partial\tau} \left(j(\tau)\right)}{j(z)-j(\tau)},
\end{align*}
which shows that $H_z(\tau)$ is a meromorphic modular form of weight $2$ for $\slz$ (see Definition \ref{defn:modularforms}). Their identity is equivalent to the denominator formula for the Monster Lie algebra.

Let $f$ be a weight $k$ meromorphic modular form (see Definition \ref{defn:modularforms}). The \emph{divisor modular form of $f$} is defined as
\begin{align*}
f^{\div} (\tau) \coloneqq \sum_{z \in \slz \backslash \H} \frac{2}{\vt{\mathrm{Stab}_z(\slz)}} \ \mathrm{ord}_z(f) H_z(\tau).
\end{align*}
More explicitly, we have (using \eqref{eq:Moebius})
\begin{align*}
\frac{2}{\vt{\mathrm{Stab}_z(\slz)}} = \begin{cases}
\frac{1}{2} & \text{if } z \sim i, \\
\frac{1}{3} & \text{if } z \sim e^{\frac{2\pi i}{3}}, \\
1 & \text{else},
\end{cases}
\end{align*}
where $z \sim z'$ indicates that $z$ and $z'$ lie in the same orbit of $\H$ under the action of $\slz$. By definition, the divisor modular form $f^{\div}$ is a weight $2$ meromorphic modular form detecting the zeros and poles of $f$. If $f$ is normalized such that its first non-zero Fourier coefficient is equal to $1$, then Bruinier, Kohnen, and Ono \cite{brukoon}*{Theorem 1.1} showed that
\begin{align} \label{eq:brukoon}
\frac{1}{2\pi i} \frac{\partial f}{\partial z} = \frac{k}{12}E_2 f - f f^{\div}, \qquad E_2(z) \coloneqq 1 - 24\sum_{n=1}^\infty \sum_{d \mid n} d q^n.
\end{align}
Recently, Bringmann, Kane, L{\"o}brich, Ono and the second author \cite{brikaloeonro} generalized the $j_n$ to $\Gamma_0(N)$ and showed that their generalizations are also related to divisor modular forms.

Consequently, it is natural to investigate $f_{k,D}^{\div}$. To this end, we introduce the quantity
\begin{align*}
Q_{z} \coloneqq \frac{1}{y}\left(a\vt{z}^2+bx+c\right) \in \R, \qquad Q = [a,b,c] \in \Qc_D,
\end{align*}
which encodes the hyperbolic geodesic 
$
S_Q \coloneqq \left\{z \in \H \colon Q_z = 0\right\}.
$
The endpoints of the geodesic $S_Q$ are the two roots of $Q(z,1)$. Since $D > 0$ throughout, these roots are contained in $\R \cup \{\infty\}$ (if $a=0$ then $b \neq 0$ and the two roots are $-\frac{c}{b}$ and $\infty$). Hence, the function
\begin{align*}
\omega_{k+1,D}(z) \coloneqq \sum_{Q \in \Qc_D} \frac{Q_{z}}{Q(z,1)^{k+1}}
\end{align*}
has no poles on $\H$ and converges absolutely for $k > 2$ (see Proposition \ref{prop:omegaconv}). 

Roughly speaking, weak Maa{\ss} forms are certain real-analytic generalizations of (complex-analytic) modular forms. In particular, they satisfy the same modular transformation laws as modular forms. Our definition of weak Maa{\ss} forms follows Bruinier and Funke \cite{brufu02}, see Definition \ref{defn:Maassforms} below. Our first main result collects the properties of $\omega_{k+1,D}$.
\begin{thm}\label{thm:OmegaFirstThm}
Let $2 < k \in 2\N$ and $D \in \N$ be a discriminant.
\begin{enumerate}[label={\rm (\roman*)}]
\item The function $\omega_{k+1,D}$ is a weak Maa{\ss} form of weight $2k+2$ for $\slz$ and eigenvalue $2k$ under the hyperbolic Laplace operator $\Delta_{2k+2}$ (defined in \eqref{eq:Deltadef}). Furthermore, it satisfies
\begin{align*}
\lim_{z \to i\infty} \omega_{k+1,D}(z) = 0.
\end{align*}
\item A splitting of $\omega_{k+1,D}$ into a holomorphic and a non-holomorphic part is given by
\begin{align*}
\omega_{k+1,D}(z) = - i\sum_{Q \in \Qc_D}\frac{Q'(z,1)}{Q(z,1)^{k+1}} + \frac{1}{y}f_{k,D}(z), \qquad Q'(z,1) \coloneqq \frac{\dd}{\dd z} Q(z,1).
\end{align*}
\item Let $c_{f_{k,D}}$ be the first non-vanishing Fourier coefficient of $f_{k,D}$ in its Fourier expansion about $i\infty$ (see Zagier \cite{zagier75}*{p.\ 45}). Let
\begin{align*}
E_2^*(z) \coloneqq E_2(z) - \frac{3}{\pi y}
\end{align*}
be the non-holomorphic Eisenstein series of weight $2$. Then, we have
\begin{align*}
\frac{1}{c_{f_{k,D}}}f_{k,D}^{\div}(z) = \frac{k}{2\pi} \frac{\omega_{k+1,D}(z)}{f_{k,D}(z)} + \frac{k}{6} E_2^*(z).
\end{align*}
\end{enumerate}
\end{thm}

\begin{rmks}
\
\begin{enumerate}
\item Further motivation for $\omega_{k+1,D}$ arises from other functions of Zagier \cite{zagier77}*{eq.\ (48)}, work of the first author \cites{mon22, mon24}, and work of Bringmann and the first author \cite{brimo}.
\item Following Zagier, \cite{zagier75}*{p.\ 44}, the function
\begin{align*}
\omega_{k+1,0}(z) = \frac{1}{\zeta(k+1)} \sum_{\gamma \in \slz_{\infty} \backslash \slz} y^{-1}\big\vert_{k+1} \gamma
\end{align*}
is a non-holomorphic Eisenstein series. Here, $\zeta$ denotes the Riemann $\zeta$-function, the Petersson slash operator is defined in \eqref{eq:slashdef}, and $\slz_{\infty}$ is defined in \eqref{eq:gaminf}.
\end{enumerate}
\end{rmks}

In the $1980$'s, Kohnen and Zagier \cites{koza81, koza84} introduced the two-variable generating function
\begin{align*}
\Omega_k(\tau,z) \coloneqq \sum_{D > 0} D^{k-\frac{1}{2}} f_{k,D}(z) e^{2\pi i D \tau}
\end{align*}
of the $f_{k,D}$. This can be viewed as a ``hyperbolic analog'' of the functions $H_z(\tau)$ from \eqref{eq:Hzdef}. Employing a result by Vign{\'e}ras \cites{vign1, vign2} (see Theorem \ref{thm:Vigneras} and the subsequent discussion), one can show that $\Omega_k(\cdot,z)$ is a cusp form of weight $k+\frac{1}{2}$ for $\Gamma_0(4)$ in Kohnen's plus space (see Section~\ref{ModularFormsSection} for the definition) for fixed $z \in \H$. Alternatively, Kohnen \cite{koh85}*{Theorem 1} proved that
\begin{align*}
\Omega_k(\tau,z) = \frac{(-1)^{\frac{k}{2}}3(2\pi)^k}{(k-1)!} \sum_{m \geq 1} m^{k-1} \sum_{d \mid m} d^k P_{k+\frac{1}{2},d^2}^{+}(\tau) q^m,
\end{align*}
where $P_{k+\frac{1}{2},d^2}^{+}$ is the weight $k+\frac{1}{2}$ holomorphic Poincar{\'e} series of exponential type in the plus space (see \eqref{eq:Pplusdef}). Kohnen and Zagier proved that $\Omega_k$, utilized as a kernel function of a theta lift, realizes the famous Shimura \cite{shim} and Shintani \cite{shin} lifts. From this, they established that the even periods of $f_{k,D}$ are rational and deduced deep results on (twisted) central $L$-values of cusp forms. More recent work on such values can be found in work of Ehlen, Guerzhoy, Kane and the second author \cite{egkr} as well as of Males, Wagner and the first two authors \cite{mmrw}.

Imitating the construction of $\Omega_k$, we define
\begin{align*}
\Lambda_k(\tau,z) \coloneqq \sum_{D > 0} D^{k-\frac{1}{2}} \omega_{k+1,D}(z) e^{2\pi i D \tau}.
\end{align*}
Our second result shows that $\Lambda_k$ shares the aforementioned modularity properties with $\Omega_k$.
\begin{thm} \label{thm:IndefThetaMain}
The function $\Lambda_k(\cdot,z)$ is modular of weight $k + \frac{1}{2}$ for $\Gamma_0(4)$ in the plus space.
\end{thm}

\begin{rmk}
In contrast to $\Omega_k$, our function $\Lambda_k$ is not holomorphic as a function of $z$.
\end{rmk}

Consequently, $\Lambda_k$ gives rise to a theta lift as well. Let $f$ be a weight $k+\frac{1}{2}$ cusp form for $\Gamma_0(4)$ satisfying Kohnen's plus space condition, and let $\langle \, \cdot \, ,\, \cdot\, \rangle$ be the usual Petersson inner product normalized by the index of the congruence subgroup (see \eqref{eq:peterssoninnerproduct}) . Then, we define the theta lift
\begin{align*}
I_{k+\frac{1}{2}}(z,f) \coloneqq \left\langle f, \Lambda_k\left(\cdot,-\overline{z}\right)\right\rangle.
\end{align*}
The Poincar{\'e} series $P_{k+\frac{1}{2},m}^{+}$ (see \eqref{eq:Pplusdef}) generate Kohnen's plus subspace within the space of weight $k+\frac{1}{2}$ cusp forms for $\Gamma_0(4)$ for varying $m \in \N$. Thus, it suffices to evaluate a theta lift on these series.
\begin{thm}\label{thm:ThetaLiftLambda}
Let $D \in \N$ be a discriminant. We have
\begin{align*}
I_{k+\frac{1}{2}}\left(z,P_{k+\frac{1}{2},D}^{+}\right) = \frac{\Gamma\left(k-\frac{1}{2}\right)}{6(4\pi)^{k-\frac{1}{2}}} \omega_{k+1,D}(z),
\end{align*}
where $\Gamma$ denotes the usual $\Gamma$-function.
\end{thm}

\begin{rmk}
This suggests that the function $\Lambda_k$ might be viewed as a dual weight analog of a theta kernel studied by Bringmann, Kane, and Viazovska \cite{BKV}*{(1.6)} as well as by Bringmann, Kane, and Zwegers \cite{BKZ}*{p.\ 751}.
\end{rmk}

Analogously to Theorem \ref{thm:ThetaLiftLambda}, we also have
\begin{align} \label{eq:pickingfkD}
\left\langle \Omega_k\left(\cdot,z\right), P_{k+\frac{1}{2},D}^{+}\right\rangle = \frac{\Gamma\left(k-\frac{1}{2}\right)}{6(4\pi)^{k-\frac{1}{2}}} f_{k,D}(z)
\end{align}
and both results can be thought of as variants of the Petersson coefficient formula. Hence, the situation can be summarized as in Figure \ref{fig:figure1}.

\begin{figure}[htbp] \label{fig:figure1}
    \begin{center}
    \begin{tikzcd}[column sep=0.75in,row sep=0.75in]
        f_{k,D}\arrow[rr, "\textup{twist by }\frac{Q_z}{Q(z,1)}"] \arrow[d, bend left=20,"\textup{gen.\ func.}"] & & \omega_{k+1,D} \arrow[d, bend left=20,"\textup{gen.\ func.}"] \\
        \Omega_k \arrow[u, bend left=20,"\eqref{eq:pickingfkD}"] \arrow[rr, "\textup{twist by }\frac{Q_z}{Q(z,1)}"]& & \Lambda_k \arrow[u, bend left=20, "\text{Thm. } \ref{thm:ThetaLiftLambda}"]
    \end{tikzcd}
    \caption{Relations between $\omega_{k+1,D}$, $f_{k,D}$, and their generating functions}
    \end{center}
\end{figure}

\noindent Another connection between $\omega_{k+1,D}$ and $f_{k,D}$ can be found in \eqref{eq:xiomega}.

This paper is organized as follows. In Section \ref{sec:Prelims}, we summarize the framework related to the topics of this paper. Section \ref{sec:ProofOmegaFirstThm} is devoted to the proof of Theorem \ref{thm:OmegaFirstThm}.  Section \ref{sec:ProofThetaLiftLambda} gives the proofs of Theorems \ref{thm:IndefThetaMain} and \ref{thm:ThetaLiftLambda}.

\section*{Acknowledgements} The authors would like to thank Winfried Kohnen, Sander Zwegers, and Caner Nazaro\u{g}lu for helpful comments.

\section{Preliminaries}\label{sec:Prelims}

\subsection{Integral binary quadratic forms} 
Let $D \in \N$ be a discriminant. The group $\slz$ acts on $\Qc_D$ by
\begin{align*}
\left(Q \circ \gamma\right)(x,y) \coloneqq Q(ax + by, cx + dy), \qquad \gamma = \begin{pmatrix}a & b \\ c & d \end{pmatrix}.
\end{align*}
This action is compatible with the usual action 
\begin{align} \label{eq:Moebius}
\gamma z \coloneqq \frac{az+b}{cz+d}, \qquad j(\gamma, z) \coloneqq cz+d,
\end{align}
of $\slz$ on $\H$ by fractional linear transformations as follows.
\begin{lemma} \label{lem:ModularBehaviour}
Let $\gamma \in \slz$, $z \in \H$, and $Q \in \Qc_D$.
\begin{enumerate}[label={\rm (\roman*)}]
\item We have $Q\circ \gamma \in \Qc_D$.
\item We have
\begin{align*}
Q(\gamma z,1) = j(\gamma,z)^{-2}(Q \circ \gamma)(z,1).
\end{align*}
\item We have
\begin{align*}
Q_{\gamma z} = (Q \circ \gamma)_z.
\end{align*}
\end{enumerate}
\end{lemma}

In addition, we require the following identities, which each can be verified by direct computation.
\begin{lemma} \label{lem:Qtricks}
Let $Q \in \Qc_D$ and $z \in \H$.
\begin{enumerate}[label={\rm (\roman*)}]
\item We have 
\begin{align*}
Dy^2+Q_{z}^2y^2 = \vt{Q(z,1)}^2.
\end{align*}
\item We have 
\begin{align*}
2iy^2 \frac{\partial}{\partial\overline{z}} Q_{z} = Q(z,1).
\end{align*}
\item We have 
\begin{align*}
\frac{\partial}{\partial\overline{z}} \frac{y^2}{Q(\overline{z},1)} = \frac{iy^2Q_{z}}{Q(\overline{z},1)^2}.
\end{align*}
\item We have 
\begin{align*}
Q_z y + iyQ'(z,1) = Q(z,1).
\end{align*}
\end{enumerate}
\end{lemma}

More details on the theory of integral binary quadratic forms can be found in Zagier’s book \cite{zagier81}, for example.

\subsection{Modular forms}\label{ModularFormsSection}
Let $\kappa \in \frac{1}{2}\Z$, $f \colon \H \to \C$ be smooth, $N \in \N$, and
\[
\Gamma_0(N) \coloneqq \left\{\begin{pmatrix} a & b \\ c & d \end{pmatrix} \in \slz \colon c \equiv 0 \pmod{N} \right\}.
\]
Let
\begin{align*}
\gamma = \begin{pmatrix} a & b \\ c & d \end{pmatrix} \in \Gamma, \qquad \Gamma \coloneqq
\begin{cases}
\slz & \text{if } \kappa \in \Z, \\
\Gamma_0(4) & \text{if } \kappa \in \Z + \frac{1}{2}.
\end{cases}
\end{align*}
Let $\left(\frac{c}{d}\right)$ be the Kronecker symbol and 
\begin{align*}
\varepsilon_d \coloneqq \begin{cases}
1 & \text{if } d \equiv 1 \pmod{4}, \\
i & \text{if } d \equiv 3 \pmod{4}
\end{cases}
\end{align*}
for every $\gamma \in \Gamma_0(4)$. The \emph{weight $\kappa$ Petersson slash operator} is given by
\begin{align} \label{eq:slashdef}
\left(f\vert_{\kappa}\gamma\right)(z) \coloneqq 
\begin{cases}
j(\gamma,z)^{-\kappa}f\left(\gamma z\right) &\text{if } \kappa \in \Z, \\[10pt]
\left(\frac{c}{d}\right) \varepsilon_{d}^{2\kappa}j(\gamma,z)^{-\kappa}f\left(\gamma z\right) & \kappa \in \Z + \frac{1}{2}.
\end{cases}
\end{align}

\begin{defn} \label{defn:modularforms}
Let $\kappa \in \frac{1}{2}\Z$ and $f \colon \H \to \C$ be smooth.
\begin{enumerate}[label=\rm(\arabic*)]
\item We call $f$ a \emph{holomorphic modular form} of weight $\kappa$ for $\Gamma$ if the following hold:
\begin{enumerate}[label=\rm(\alph*)]
\item For $\gamma \in \Gamma$ we have $f\vert_{\kappa}\gamma = f$,
\item we have that $f$ is holomorphic on $\H$,
\item we have that $f$ is holomorphic at all cusps of $\Gamma$.
\end{enumerate}
We denote the vector space of functions satisfying these conditions by $M_{\kappa}(\Gamma)$.
\item If $f$ vanishes at all cusps in addition, then we call $f$ a \emph{cusp form}. The space of cusp forms is denoted by $S_{\kappa}(\Gamma)$.
\item If $f$ satisfies the conditions (a) and (b) from (1) and is permitted to have a pole at the cusps, then we call $f$ a \emph{weakly holomorphic modular form} of weight ${\kappa}$. The vector space of such functions is denoted by $M_{\kappa}^!(\Gamma)$.
\item If $f$ satisfies the condition (a) from (1) but is permitted to have a pole at $i\infty$ as well as on $\H$, then we call $f$ a \emph{meromorphic modular form} of weight ${\kappa}$.
\end{enumerate}
We say that $f$ is in \emph{Kohnen's plus space} if $\kappa \in \Z + \frac{1}{2}$ and its Fourier coefficients (expanded about $i\infty$) are supported on discriminants (that is on indices $n$ satisfying $(-1)^{\kappa-\frac{1}{2}} n \equiv 0, 1 \pmod{4}$).
\end{defn}

We construct an example of (weakly) holomorphic modular forms. Let $m \in \Z$ and recall that $2 < k \in 2\N$. We note that the stabilizer of infinity is 
\begin{align} \label{eq:gaminf}
\Gamma_{\infty} = \left\{\pm\begin{pmatrix} 1 & n \\ 0 & 1 \end{pmatrix} \colon n \in \Z \right\},
\end{align}
and we define the classical \emph{Poincar{\'e} series of exponential type} by
\begin{align} \label{eq:Pdef}
P_{k,m}(z) \coloneqq \sum_{\gamma \in \Gamma_{\infty} \backslash \Gamma} e^{2\pi i m z} \big\vert_{k}\gamma.
\end{align}
It turns out that
\begin{align*}
P_{k,m} \in \begin{cases}
S_{k}\left(\slz\right) &\text{if } m > 0, \\
M_{k}\left(\slz\right) &\text{if } m = 0, \\
M_{k}^!\left(\slz\right) &\text{if } m < 0,
\end{cases}
\end{align*}
and $P_{k,m}$ spans $S_{k}\left(\slz\right)$ resp. $M_{k}^!\left(\slz\right)$ for $m \neq 0$ (see \cite{thebook}*{Theorems 6.8 and 6.9}). 

We let $\mathrm{pr}^+$ be Kohnen's projection operator to the plus space, which can be found in \cite{koh85}*{Proposition 3}. Then, we analogously define
\begin{align} \label{eq:Pplusdef}
P_{k+\frac{1}{2},m}^{+}(z) \coloneqq \mathrm{pr}^+ \sum_{\gamma \in \Gamma_{\infty} \backslash \Gamma} e^{2\pi i m z} \big\vert_{k+\frac{1}{2}}\gamma.
\end{align}
Adapting the proof for the integral weight case, one can show that $P_{k+\frac{1}{2},m}^{+}$ generates the plus space within $S_{k+\frac{1}{2}}\left(\Gamma_0(4)\right)$ for varying $m \in \N$.

More details on modular forms and their variants can be found in \cites{cohstr, the123}, for instance.

\subsection{Maa{\ss} forms}
We also require certain non-holomorphic modular forms. To introduce them, we define the \emph{Bruinier--Funke operator} \cite{brufu02}
\begin{align*}
\xi_{\kappa} \coloneqq 2iy^{\kappa} \overline{\frac{\partial}{\partial\overline{z}}},
\end{align*}
as well as the \emph{weight $\kappa$ hyperbolic Laplace operator}
\begin{align} \label{eq:Deltadef}
\Delta_{\kappa} \coloneqq - \xi_{2-\kappa}\xi_{\kappa} = -y^2 \left(\frac{\partial^2}{\partial x^2} + \frac{\partial^2}{\partial y^2}\right) + i\kappa y \left(\frac{\partial}{\partial x} + i\frac{\partial}{\partial y}\right).
\end{align}

\begin{defn} \label{defn:Maassforms}
Let  $f \colon \H \to \C$ be smooth. The function $f$ is called a \emph{weight $\kappa$ weak Maa{\ss} form of eigenvalue $\lambda$ for $\Gamma$} if
\begin{enumerate}[label={\rm (\alph*)}]
    \item we have $\left(f\vert_{\kappa}\gamma\right)(z) = f(z)$ for all $\gamma \in \Gamma$ and all $z \in \H$,
    \item we have $\Delta_{\kappa} f = \lambda f$ on $\H$,
    \item the function $f$ is of at most linear exponential growth towards all cusps.
\end{enumerate}
We say that $f$ is in Kohnen's plus space if $f$ satsifies the corresponding condition at the end of Definition \ref{defn:modularforms}.
\end{defn}

Let $f$, $g \colon \H \to \C$ be smooth and both transforming like modular forms of weight $\kappa \in \frac{1}{2}\Z$. Let $\dd\mu(z) \coloneqq \frac{\dd x\dd y}{y^2}$ be the usual hyperbolic measure on $\H$. Suppose that $f$ or $g$ vanish at all cusps. Then the \emph{Petersson inner product} of $f$ and $g$ is defined as
\begin{align} \label{eq:peterssoninnerproduct}
\langle f, g \rangle \coloneqq \frac{1}{\left[\slz \colon \Gamma\right]}\int_{\Gamma \backslash \H} f(z)\overline{g(z)}y^{k}\dd\mu(z),
\end{align}
where we have normalized it such that it does not depend on the choice of $\Gamma$. If $\Gamma = \Gamma_0(4)$ then the index of $\Gamma_0(4)$ in $\slz$ equals $6$ (see \cite{cohstr}*{Corollary 6.2.13} for instance). 

More details on the theory behind weak Maa{\ss} forms and variants thereof can be found in \cite{thebook}, for example.

\subsection{Indefinite theta functions}
Let $L \subseteq \R^n$ be a lattice. We define the \emph{Euler operator} 
\begin{align*}
    E \coloneqq \sum_{j=1}^nw_j\frac{\partial}{\partial w_j}, \qquad w \in L.
\end{align*}
We denote the Gram matrix associated to a non-degenerate quadratic form $\mathfrak{q}$ on $\R^n$ by $A$. The \emph{Laplace operator associated to $\mathfrak{q}$} is then defined by 
\begin{align*}
    \Delta^{(A)} \coloneqq \left\langle\frac{\partial}{\partial w},A^{-1}\frac{\partial}{\partial w}\right\rangle.
\end{align*}
Here, $\langle \, \cdot \, ,\, \cdot\, \rangle$ denotes the standard Euclidean inner product on $\R^n$.

\begin{thm}[Vign\'eras \cite{vign2}*{Th\'{e}or\`{e}me 1, p. 227}]\label{thm:Vigneras}
Suppose that $n\in \N$, $\mathfrak{q}$ is a nondegenerate quadratic form on $\R^n$, $L\subset \R^n$ is a lattice on which $\mathfrak{q}$ takes integer values, and $p \colon \R^n \to \C$ is a function satisfying the following conditions:
\begin{enumerate}[label={\rm (\roman*)}]
\item The function $f(w)\coloneqq p(w)e^{-2\pi \mathfrak{q}(w)}$ times any polynomial of degree at most $2$ and all partial derivatives of $f$ of order at most $2$ are elements of $L^2\left(\R^n\right)\cap L^1\left(\R^n\right)$.
\item For some $\lambda\in\Z$, the function $p$ satisfies
\begin{align*}
\left(E-\frac{\Delta^{(A)}}{4\pi}\right)p = \lambda p. 
\end{align*}
\end{enumerate}
Then the indefinite theta function ($\tau = u+iv$)
\begin{align*}
v^{-\frac{\lambda}{2}}\sum_{w\in L} p\left(\sqrt{v}w\right)e^{2\pi i \mathfrak{q}(w) \tau}
\end{align*}
transforms like a modular form of weight $\lambda+\frac{n}{2}$ for $\Gamma_0(N)$ and character $\chi\cdot \chi_{-4}^{\lambda}$, where $N$ and $\chi$ are the level and character of $\mathfrak{q}$ and $\chi_{-4}$ is the unique primitive Dirichlet character of conductor $4$.
\end{thm}

\begin{rmks}
\
\begin{enumerate}
        \item Note that $p$ has to satisfy both conditions globally. In other words, we require $p \in C^2\left(\R^n\right)$, namely $p$ is twice continuously differentiable on $\R^n$.
        \item Vign\'eras remarks that if $p$ satisfies $Ep = \lambda p$ and $\Delta^{(A)} p = 0$, then the theta series
        \[
			\sum_{w \in L} p(w) e^{2\pi i \mathfrak{q}(w) \tau}
		\]
        is holomorphic on $\H$. However, it may exhibit singularities at the cusps.
    \end{enumerate}
\end{rmks}

Following Vign\'{e}ras, a special case of (2) is (correcting a typo),
        \[
			p(w) = \frac{\mathfrak{q}(w)^{k-1}}{\langle w, s \rangle_{\mathfrak{q}}^{k+\frac{n}{2}-2}}
		\]
        where $k$ is any integer, $\langle \, \cdot \, ,\, \cdot\, \rangle_{\mathfrak{q}}$ is the bilinear form derived from $\mathfrak{q}$ and $s \in \C^n$ satisfies $\mathfrak{q}(s) = 0$ as well as $\langle s, \overline{s}\rangle_{\mathfrak{q}} > 0$. The function $\Omega_k(\tau,z)$ is an example of this situation by choosing $\mathfrak{q}(a,b,c) = b^2 - 4ac$, $s = \left(\frac{1}{2}, z, \frac{z^2}{2}\right)$ and $p(a,b,c) = 0$ for $\mathfrak{q}(a,b,c) < 0$. Note that
        \[
			\langle s, \overline{s}\rangle_{\mathfrak{q}} = \vt{z}^2 - \re\left(z^2\right) > 0.
		\]

\section{Proof of Theorem \ref{thm:OmegaFirstThm}}\label{sec:ProofOmegaFirstThm}

We begin by establishing that $\omega_{k+1,D}$ converges absolutely for $k > 2$.
\begin{prop} \label{prop:omegaconv}
If $k > 2$ then the function $\omega_{k+1,D}$ converges absolutely on $\H$ and uniformly towards $i\infty$.
\end{prop}

\begin{proof}
If $k$ is odd then $\omega_{k+1,D}$ vanishes identically by $Q \mapsto -Q$ and $(-Q)_z = -Q_z$. Hence, we suppose that $k \geq 4$ is even. By Lemma \ref{lem:Qtricks} (i), we infer
\begin{align*}
\vt{Q_{z}} = \sqrt{Q_{z}^2} = \sqrt{\frac{\vt{Q(z,1)}^2}{y^2} - D} \leq \sqrt{\frac{\vt{Q(z,1)}^2}{y^2}} = \frac{\vt{Q(z,1)}}{y}.
\end{align*}
Hence, we estimate
\begin{align} \label{eq:omegaestimate}
\vt{\omega_{k+1,D}(z)} \leq \sum_{Q \in \Qc_D} \frac{\vt{Q_{z}}}{\vt{Q(z,1)}^{k+1}} \leq \frac{1}{y} \sum_{Q \in \Qc_D} \frac{1}{\vt{Q(z,1)}^{k}}.
\end{align}
Since $f_{k,D}$ converges absolutely on $\H$ and uniformly towards $i\infty$ for $k > 2$ (see \cite{zagier75}*{p.\ 3}), the claim follows.
\end{proof}

Now, we are in position to prove our first main result.
\begin{proof}[Proof of Theorem \ref{thm:OmegaFirstThm}]
\
\begin{enumerate}[label={\rm (\roman*)}]
\item Modularity follows by Lemma \ref{lem:ModularBehaviour}. Indeed, let $\gamma \in \slz$. Then, we have
\begin{align*}
\hspace*{\leftmargini} \omega_{k+1,D}(\gamma z) &= \sum_{Q \in \Qc_D} \frac{Q_{\gamma z}}{Q(\gamma z,1)^{k+1}} = \sum_{Q \in \Qc_D} \frac{(Q \circ \gamma)_{z}}{j(\gamma,z)^{-2k-2}(Q\circ\gamma)(z,1)^{k+1}} \\
&= j(\gamma,z)^{2k+2}\omega_{k+1,D}(z),
\end{align*}
where we used absolute convergence in the last step. Next, we compute the eigenvalue under the Laplace operator. By Lemma \ref{lem:Qtricks} (ii), we have
\begin{align*}
\frac{\partial}{\partial\overline z} \frac{Q_{ z}}{Q( z,1)^{k+1}} =
\frac{1}{Q( z,1)^{k+1}} \frac{\partial}{\partial\overline z}Q_{ z} = \frac{1}{2iy^2} \frac{1}{Q( z,1)^{k}},
\end{align*}
and hence
\begin{align} \label{eq:xiomega}
\xi_{2k+2}\left(\frac{Q_{z}}{Q(z,1)^{k+1}}\right) = -\frac{y^{2k}}{Q(\overline{z},1)^{k}}.
\end{align}
Moreover, Lemma \ref{lem:Qtricks} (iii) yields
\begin{align*}
\frac{\partial}{\partial\overline{z}} \frac{y^{2k}}{Q(\overline{z},1)^{k}} = \frac{\partial}{\partial\overline{z}} \left(\frac{y^2}{Q(\overline{z},1)}\right)^k = k \left(\frac{y^2}{Q(\overline{z},1)}\right)^{k-1} \frac{iy^2Q_{z}}{Q(\overline{z},1)^2} = ik\frac{y^{2k}Q_{z}}{Q(\overline{z},1)^{k+1}},
\end{align*}
and we infer ($Q_z \in \R$)
\begin{align*}
\xi_{-2k} \left(\frac{y^{2k}}{Q(\overline{z},1)^{k}}\right) = 2k \frac{Q_{z}}{Q(z,1)^{k+1}}.
\end{align*}
By absolute convergence, we may argue termwise getting
\begin{align*}
\hspace*{\leftmargini} \Delta_{2k+2}\omega_{k+1,D}(z) = -\xi_{-2k}\xi_{2k+2}\omega_{k+1,D}(z) = \sum_{Q\in\Qc_D} \xi_{-2k} \frac{y^{2k}}{Q(\overline{z},1)^{k}} = 2k\omega_{k+1,D}(z).
\end{align*}
The last assertion follows by \eqref{eq:omegaestimate} and the fact that $f_{k,D}$ is a cusp form.

\item The decomposition follows directly by Lemma \ref{lem:Qtricks} (iv).

\item We define
\begin{align*}
g(z) \coloneqq \frac{f_{k,D}(z)}{c_{f_{k,D}}}
\end{align*}
By \cite{brukoon}*{Theorem 1.1} (see \eqref{eq:brukoon}), we have
\begin{align*}
g^{\div} (z) = -\frac{1}{2\pi i} \frac{\frac{\dd}{\dd z}g(z)}{g(z)} + \frac{2k}{12}E_2(z) = -\frac{1}{2\pi i} \frac{\frac{\dd}{\dd z}f_{k,D}(z)}{f_{k,D}(z)} + \frac{k}{6}E_2(z).
\end{align*}
Moreover, part (ii) yields
\begin{align*}
\frac{\dd}{\dd z}f_{k,D}(z) &= -k \sum_{Q \in \Qc_D} \frac{Q'(z,1)}{Q(z,1)^{k+1}} = \frac{ik}{y}f_{k,D}(z) - ik\omega_{k+1,D}(z),
\end{align*}
from which we infer
\begin{align*}
g^{\div} (z) & = -\frac{1}{2\pi i} \left(\frac{ik}{y} - ik\frac{\omega_{k+1,D}(z)}{f_{k,D}(z)}\right) + \frac{k}{6}E_2(z).
\end{align*}
The claim follows by the relation
\begin{align*}
g^{\div} (z) = \frac{f_{k,D}^{\div}(z)}{c_{f_{k,D}}}
\end{align*}
and the definition of $E_2^*$. \qedhere
\end{enumerate}
\end{proof}

\section{Proofs of Theorems \ref{thm:IndefThetaMain} and \ref{thm:ThetaLiftLambda}} \label{sec:ProofThetaLiftLambda}

\begin{proof}[Proof of Theorem \ref{thm:IndefThetaMain}]
We will employ Theorem \ref{thm:Vigneras}. To this end, we choose $N=4$, $L = \Z^3$ and $\mathfrak{q}(a,b,c) \coloneqq b^2-4ac$ with associated Gram matrix
\begin{align*}
A = \begin{pmatrix} 0 & 0 & -4 \\ 0 & 2 & 0 \\ -4 & 0 & 0 \end{pmatrix}, \qquad
A^{-1} = \begin{pmatrix} 0 & 0 & -\frac{1}{4} \\ 0 & \frac{1}{2} & 0 \\ -\frac{1}{4} & 0 & 0 \end{pmatrix}.
\end{align*}
We define
\begin{align*}
p(a,b,c) \coloneqq \begin{cases} \left(b^2-4ac\right)^{k-\frac{1}{2}}\frac{\frac{1}{y}\left(a\vt{z}^2+bx+c\right)}{\left(az^2+bz+c\right)^{k+1}} & \text{if } D = b^2-4ac > 0, \\
0 & \text{if } D = b^2-4ac \leq 0.
\end{cases}
\end{align*}
Then $p$ is twice continuously differentiable at $D=0$ and has no singularities as a function of $(a,b,c)$ since $z \in \H$, which implies that $p \in C^2\left(\R^3\right)$. The first condition in Vign\'eras' theorem and absolute convergence of $\Lambda_k$ follow by the exponential decay of $\omega_{k+1,D}$ towards all cusps, see Theorem \ref{thm:OmegaFirstThm} (i), as well as combining \eqref{eq:omegaestimate} with absolute convergence of Kohnen and Zagier's kernel $\Omega_k$. Moreover, we have
\begin{align*}
p\left(\sqrt{v}a,\sqrt{v}b,\sqrt{v}c\right) = v^{k-\frac{1}{2}+\frac{1}{2}-\frac{k+1}{2}}p(a,b,c) = v^{\frac{k-1}{2}}p(a,b,c).
\end{align*}
Noting that 
\begin{align*}
\frac{\partial}{\partial a}\frac{\partial}{\partial c} p = \frac{\partial}{\partial c}\frac{\partial}{\partial a} p
\end{align*}
by Schwarz's theorem, we have
\begin{align*}
E = a\frac{\partial}{\partial a}+b\frac{\partial}{\partial b}+c\frac{\partial}{\partial c}, \qquad 
\Delta^{(A)} = - \frac{1}{2}\frac{\partial}{\partial a}\frac{\partial}{\partial c} + \frac{1}{2}\frac{\partial^2}{\partial b^2}.
\end{align*}
To verify the second condition, we need to show that $p(a,b,c)$ is an eigenfunction of $E-\frac{\Delta^{(A)}}{4\pi}$ with eigenvalue $\lambda = k-1$. This can be verified with SAGE \cite{sage} for example, which we defer to Appendix \ref{app:appendix}.

Thus, Theorem \ref{thm:Vigneras} applies. By construction, we observe that
\begin{align*}
v^{-\frac{\lambda}{2}} \sum_{w \in \Z^3} p\left(\sqrt{v}w\right) e^{2\pi i q(w)\tau} = \Lambda_k(\tau,z),
\end{align*}
which completes the proof.
\end{proof}

It remains to prove Theorem \ref{thm:ThetaLiftLambda}.
\begin{proof}[Proof of Theorem \ref{thm:ThetaLiftLambda}]
This follows by the Petersson coefficient formula. We provide a proof for convenience of the reader. As stated in \cite{BKV}*{(2.5)}, $\mathrm{pr}^+$ is Hermitian with respect to the Petersson inner product. Since $k$ is even, $\Lambda_k(\tau,z)$ satisfies the plus space condition as a function of $\tau$. Hence, we obtain
\begin{align*}
\left\langle P_{k+\frac{1}{2},D}^{+}, \Lambda_k\left(\cdot,-\overline{z}\right)\right\rangle = \frac{1}{6} \int_{\Gamma_0(4) \backslash \H} P_{k+\frac{1}{2}, D}(\tau) \overline{\Lambda_k\left(\tau,-\overline{z}\right)}  v^{k+\frac{1}{2}} \dd\mu(\tau).
\end{align*}
Writing $Q = [a,b,c]$, we have
\begin{align*}
\overline{\left([a,b,c]_{-\overline{z}}\right)} = [a,b,c]_{-\overline{z}} = [a,-b,c]_z, \qquad \overline{\left([a,b,c]\left(-\overline{z},1\right)\right)} = [a,b,c](-z,1) = [a,-b,c](z,1).
\end{align*}
Hence, we map $b \mapsto -b$ in each sum running over $Q = [a,b,c] \in \Qc_D$. According to Theorem \ref{thm:IndefThetaMain}, $\tau \mapsto \Lambda_k\left(\tau,-\overline{z}\right)$ is modular of weight $k+\frac{1}{2}$ for $\Gamma_0(4)$. Unfolding and inserting the definitions of both functions subsequently yields
\begin{align*}
\left\langle P_{k+\frac{1}{2},D}^{+}, \Lambda_k\left(\cdot,-\overline{z}\right)\right\rangle = \frac{1}{6} \sum_{d > 0} d^{k-\frac{1}{2}}\omega_{k+1,d}(z) \int_{0}^{\infty} \int_{0}^{1} e^{2 \pi i D\tau} e^{-2\pi i d \overline{\tau}} v^{k+\frac{1}{2}} \frac{\dd u\dd v}{v^2}.
\end{align*}
As usual, the integral over $u$ picks the $D$-th coefficient. That is, we infer
\begin{align*}
\left\langle P_{k+\frac{1}{2},D}^{+}, \Lambda_k\left(\cdot,-\overline{z}\right)\right\rangle = \frac{1}{6} D^{k-\frac{1}{2}}\omega_{k+1,D}(z) 
\int_{0}^{\infty} v^{k+\frac{1}{2}} e^{-4 \pi D v} \frac{\dd v}{v^2},
\end{align*}
and the claim follows directly.
\end{proof}

\appendix

\section{SAGE code} \label{app:appendix}

We provide the SAGE \cite{sage} code used in the proof of Theorem \ref{thm:IndefThetaMain}.

\begin{figure}[htbp]
	\begin{lstlisting}[language=Sage]
	sage: var('a', 'b', 'c', 'k', 'x','y');
	sage: p = (b^2-4*a*c)^(k-1/2)*(1/y*(a*(x^2+y^2)+b*x+c))/
	....: (a*(x+I*y)^2+b*(x+I*y)+c)^(k+1);
	sage: euler = a*p.derivative(a)+b*p.derivative(b)
	....: +c*p.derivative(c);
	sage: delt = -1/2*(p.derivative(c)).derivative(a)
	....: +1/2*(p.derivative(b)).derivative(b);
	sage: vign = euler-1/(4*pi)*delt;
	sage: (vign/p).full_simplify()
	k - 1
	\end{lstlisting}
	\caption{Verification of Vign{\'e}ras' differential equation}
\end{figure}


\begin{bibsection}
\begin{biblist}
\bib{askani}{article}{
   author={Asai, T.},
   author={Kaneko, M.},
   author={Ninomiya, H.},
   title={Zeros of certain modular functions and an application},
   journal={Comment. Math. Univ. St. Paul.},
   volume={46},
   date={1997},
   number={1},
   pages={93--101},
}

\bib{beng13}{thesis}{
   author={Bengoechea, P.},
   title={Corps quadratiques et formes modulaires},
   type={Ph.D. Thesis},
   organization={Universit{\'e} Pierre et Marie Curie},
   date={2013},
}

\bib{thebook}{book}{
    AUTHOR = {Bringmann, K.},
    AUTHOR = {Folsom, A.},
    AUTHOR = {Ono, K.},
    AUTHOR = {Rolen, L.},
     TITLE = {Harmonic {M}aass forms and mock modular forms: theory and
              applications},
    SERIES = {American Mathematical Society Colloquium Publications},
    VOLUME = {64},
 PUBLISHER = {American Mathematical Society, Providence, RI},
      YEAR = {2017},
     PAGES = {xv+391},
}

\bib{brikaloeonro}{article}{
   author={Bringmann, K.},
   author={Kane, B.},
   author={L\"{o}brich, S.},
   author={Ono, K.},
   author={Rolen, L.},
   title={On divisors of modular forms},
   journal={Adv. Math.},
   volume={329},
   date={2018},
   pages={541--554},
}

\bib{BKV}{article}{
   author={Bringmann, K.},
   author={Kane, B.},
   author={Viazovska, M.},
   title={Theta lifts and local Maass forms},
   journal={Math. Res. Lett.},
   volume={20},
   date={2013},
   number={2},
   pages={213--234},
}

\bib{BKZ}{article}{
   author={Bringmann, K.},
   author={Kane, B.},
   author={Zwegers, S.},
   title={On a completed generating function of locally harmonic Maass
   forms},
   journal={Compos. Math.},
   volume={150},
   date={2014},
   number={5},
   pages={749--762},
}

\bib{brimo}{webpage}{
   author={Bringmann, K.},
   author={Mono, A.},
   title={A modular framework of functions of Knopp and indefinite binary quadratic forms},
   url={http://arxiv.org/abs/2208.01451},
   note={preprint},
   year={2022},
}

\bib{brufu02}{article}{
   author={Bruinier, J. H.},
   author={Funke, J.},
   title={On two geometric theta lifts},
   journal={Duke Math. J.},
   volume={125},
   date={2004},
   number={1},
   pages={45--90},
}

\bib{the123}{collection}{
   author={Bruinier, J. H.},
   author={van der Geer, G.},
   author={Harder, G.},
   author={Zagier, D.},
   title={The 1-2-3 of modular forms},
   series={Universitext},
   note={Lectures from the Summer School on Modular Forms and their
   Applications held in Nordfjordeid, June 2004;
   Edited by Kristian Ranestad},
   publisher={Springer-Verlag, Berlin},
   date={2008},
}

\bib{brukoon}{article}{
   author={Bruinier, J. H.},
   author={Kohnen, W.},
   author={Ono, K.},
   title={The arithmetic of the values of modular functions and the divisors
   of modular forms},
   journal={Compos. Math.},
   volume={140},
   date={2004},
   number={3},
   pages={552--566},
}

\bib{cohstr}{book}{
   author={Cohen, H.},
   author={Str\"{o}mberg, F.},
   title={Modular forms},
   series={Graduate Studies in Mathematics},
   volume={179},
   note={A classical approach},
   publisher={American Mathematical Society, Providence, RI},
   date={2017},
   pages={xii+700},
}

\bib{egkr}{article}{
   author={Ehlen, S.},
   author={Guerzhoy, P.},
   author={Kane, B.},
   author={Rolen, L.},
   title={Central $L$-values of elliptic curves and local polynomials},
   journal={Proc. Lond. Math. Soc. (3)},
   volume={120},
   date={2020},
   number={5},
   pages={742--769},
}

\bib{IOS}{article}{
   author={Imamo\={g}lu, \"{O}.},
   author={O'Sullivan, C.},
   title={Parabolic, hyperbolic and elliptic Poincar\'{e} series},
   journal={Acta Arith.},
   volume={139},
   date={2009},
   number={3},
   pages={199--228},
}

\bib{katok}{article}{
   author={Katok, S.},
   title={Closed geodesics, periods and arithmetic of modular forms},
   journal={Invent. Math.},
   volume={80},
   date={1985},
   number={3},
   pages={469--480},
}

\bib{koh85}{article}{
   author={Kohnen, W.},
   title={Fourier coefficients of modular forms of half-integral weight},
   journal={Math. Ann.},
   volume={271},
   date={1985},
   number={2},
   pages={237--268},
}

\bib{koza81}{article}{
   author={Kohnen, W.},
   author={Zagier, D.},
   title={Values of $L$-series of modular forms at the center of the
   critical strip},
   journal={Invent. Math.},
   volume={64},
   date={1981},
   number={2},
   pages={175--198},
}

\bib{koza84}{article}{
   author={Kohnen, W.},
   author={Zagier, D.},
   title={Modular forms with rational periods},
   conference={
      title={Modular forms},
      address={Durham},
      date={1983},
   },
   book={
      series={Ellis Horwood Ser. Math. Appl.: Statist. Oper. Res.},
      publisher={Horwood, Chichester},
   },
   date={1984},
   pages={197--249},
}

\bib{mmrw}{webpage}{
   author={Males, J.},
   author={Mono, A.},
   author={Rolen, L.},
   author={Wagner, I.},
   title={Central $L$-values of newforms and local polynomials},
   year={2023},
   url={https://arxiv.org/abs/2306.15519},
   note={preprint},
}

\bib{mon22}{article}{
	author={Mono, A.},
	title={Eisenstein series of even weight $k\geq2$ and integral binary
		quadratic forms},
	journal={Proc. Amer. Math. Soc.},
	volume={150},
	date={2022},
	number={5},
	pages={1889--1902},
}

\bib{mon24}{article}{
   author={Mono, A.},
   title={Locally harmonic Maass forms of positive even weight},
   journal={Israel J. Math.},
   volume={261},
   date={2024},
   number={2},
   pages={671--694},
}

\bib{pet43}{article}{
   author={Petersson, H.},
   title={Ein Summationsverfahren f\"ur die Poincar\'eschen Reihen von der
   Dimension --2 zu den hyperbolischen Fixpunktepaaren},
   language={German},
   journal={Math. Z.},
   volume={49},
   date={1944},
   pages={441--496},
}

\bib{shim}{article}{
   author={Shimura, G.},
   title={On modular forms of half integral weight},
   journal={Ann. of Math. (2)},
   volume={97},
   date={1973},
   pages={440--481},
}

\bib{shin}{article}{
   author={Shintani, T.},
   title={On construction of holomorphic cusp forms of half integral weight},
   journal={Nagoya Math. J.},
   volume={58},
   date={1975},
}

\bib{vign1}{article}{
   author={Vign\'{e}ras, M.-F.},
   title={S\'{e}ries th\^{e}ta des formes quadratiques ind\'{e}finies},
   language={French},
   conference={
      title={S\'{e}minaire Delange-Pisot-Poitou, 17e ann\'{e}e (1975/76), Th\'{e}orie
      des nombres: Fasc. 1, Exp. No. 20},
   },
   book={
      publisher={Secr\'{e}tariat Math., Paris},
   },
   date={1977},
   pages={3},
}

\bib{vign2}{article}{
    author= {Vign\'{e}ras, M.-F.},
    title={S\'{e}ries th\^{e}ta des formes quadratiques ind\'{e}finies},
    book={
        title={Modular Functions of One Variable VI. Lecture Notes in Mathematics},
        volume={627},
        editor={Serre, J.-P.},
        editor={Zagier, D.},
        publisher={Springer},
        place={Berlin, Heidelberg},
        year={1977}
        },
    pages={227--239}
}

\bib{zagier75}{article}{
	author={Zagier, D.},
	title={Modular forms associated to real quadratic fields},
	journal={Invent. Math.},
	volume={30},
	date={1975},
	number={1},
	pages={1--46},
}

\bib{zagier77}{article}{
   author={Zagier, D.},
   title={Modular forms whose Fourier coefficients involve zeta-functions of quadratic fields},
   book={
      series={Modular Forms of One Variable},
      volume={VI},
      publisher={Springer},
   },
   date={1977},
   pages={105--169},
}

\bib{zagier81}{book}{
   author={Zagier, D.},
   title={Zetafunktionen und quadratische K\"{o}rper},
   language={German},
   series={Hochschultext [University Textbooks]},
   note={Eine Einf\"{u}hrung in die h\"{o}here Zahlentheorie. [An introduction to
   higher number theory]},
   publisher={Springer-Verlag, Berlin-New York},
   date={1981},
}

\bib{zagier02}{article}{
   author={Zagier, D.},
   title={Traces of singular moduli},
   conference={
      title={Motives, polylogarithms and Hodge theory, Part I},
      address={Irvine, CA},
      date={1998},
   },
   book={
      series={Int. Press Lect. Ser.},
      volume={3, I},
      publisher={Int. Press, Somerville, MA},
   },
   isbn={1-57146-090-X},
   date={2002},
   pages={211--244},
}

\bib{sage}{misc}{
  author = {W. A. Stein et al.},
  title = {Sage Mathematics Software},
  note = {The Sage Development Team, Version 9.3, \url{https://www.sagemath.org/}},
  year= {2022},
}
\end{biblist}
\end{bibsection}
\end{document}